\documentclass[a4paper,draft,reqno,12pt]{amsart}
\usepackage[english]{babel}
\usepackage{amsmath}
\usepackage{amssymb}
\usepackage{amscd}
\usepackage{amsthm}
\usepackage{euscript}
\usepackage{tikz}
\newtheorem{prop}{Proposition}

\newtheorem{theor}{Theorem}

\newtheorem{cor}{Corollary}
\theoremstyle{definition}
\newtheorem{de}{Definition}
\newtheorem{ex}{Example}
\theoremstyle{remark}
\newtheorem {re}{Remark}

\DeclareMathOperator{\Aut}{Aut}

\def\BG{{\mathbb G}}

\def\BG{{\mathbb G}}

\def\BC{{\mathbb C}}

\def\BZ{{\mathbb Z}}

\def\BN{{\mathbb N}}
\def\BQ{{\mathbb Q}}
\def\BP{{\mathbb P}}

\def\Cl{\mathrm{Cl}}

\begin{document}

\date{}
\title[Euler-symmetric projective toric varieties and additive actions]{Euler-symmetric projective toric varieties and additive actions}
\author{Anton Shafarevich}
\address{Moscow Center of Fundamental and Applied Mathematics,  Moscow, Russia; \linebreak and \linebreak
National Research University Higher School of Economics, Faculty of Computer Science, Pokrovsky Boulevard 11, Moscow, 109028, Russia}
\email{shafarevich.a@gmail.ru}

\thanks{The author was supported by the grant RSF-19-11-00172.}

\maketitle

\begin{abstract}
Let  $\BG_a$ be the additive group of the field of complex numbers $\BC$. We say that an irreducible algebraic variety $X$ of dimension $n$ admits an additive action if there is a regular action of the group $\BG_a^n = \BG_a \times \ldots \times \BG_a$ ($n$ times) on  $X$ with an open orbit. In 2017 Baohua Fu and Jun-Muk Hwang introduced a class of Euler-symmetric varieties. They gave a classification of Euler-symmetric varieties and proved that any Euler-symmetric variety admits an additive action. In this paper we show that in the case of projective toric varieites the converse is also true. More precisely, a projective toric variety admitting an additive action is an Euler-symmetric variety with respect to any linearly normal embedding into a projective space. Also we discuss some properties of Euler-symmetric projective toric varieties. \end{abstract}

\section{Introduction}

Let $X \subseteq \mathbb{P}^s$ be a projective variety over the field $\mathbb{C}$. Suppose that $X$ is nondegenerate, that is, $X$ is not contained in any hyperplane in $\mathbb{P}^s$. Denote by $\mathbb{G}_m$ the group $\mathbb{C}^{*}$. The following definition was given in \cite{FH}.

\begin{de}
Let $x\in X$ be a smooth point. A $\mathbb{G}_m$-action on $\mathbb{P}^s$ is said to be of \emph{Euler type} at $x$, if the following conditions hold:

\begin{enumerate}
\item the variety $X$ is invariant with respect to this action;

\item the point $x$ is isolated fixed point in $X$ with respect to this action;

\item the induced action on the tangent space $T_x X$ acts by scalar operators.

\end{enumerate}

 We say that $x$ is an \emph{Euler} point if there is a $\mathbb{G}_m$-action on $\mathbb{P}^s$ of Euler type at $x$. The variety $X$ is called \emph{Euler-symmetric} if there is an open subset $U \subseteq X$ such that every point of $U$ is an Euler point.

\end{de}

In \cite{FH} Baohua Fu and Jun-Muk Hwang gave a description of Euler-symmetric varieties using their fundamental forms at a general point. Also they proved that any Euler-symmetric variety admits an additive action \cite[Theorem 3.7]{FH}.

There are several results on additive actions on complete toric varieties. The first one is the work of Hassett and Tschinkel \cite{HT}. They established a correspondence between additive actions on the projective space $\mathbb{P}^n$ and local commutative associative algebras with unit of dimension $n+1$. One can find in \cite{KL} a more general correspondence. 

Additive actions on projective hypersurfaces are studying in \cite{AP} and \cite{AS}. One can find results on additive actions on flag varieties in \cite{A}, \cite{D} and \cite{F}.  Also there are works on additive actions on singular del Pezzo surfaces \cite{DL}, weighted projective planes \cite{ABZ} and Hirzebruch surfaces \cite{HT}.

We recall that a \emph{toric variety} is a normal variety of dimension $n$ which admits an effective action of torus $T = \mathbb{G}_m^n = \mathbb{G}_m \times \ldots \times \mathbb{G}_m$ ($n$ times) with an open orbit. Due to a combinatorical description (see \cite{FULTON} or \cite{COX}) a lot of properties of toric varieties can be described in a nice and simple way. So it is always natural to consider the toric case while studying some general theory.

Toric varieties admitting additive actions are described in \cite{ARRO}. It was proven in \cite{DZHUNUS1} that a complete toric variety of dimension two which admits an additive action can have either one or two non-isomorphic additive actions.  There is a description of complete toric varieties with a unique additive action \cite{DZHUNUS2}. In \cite{S}  projective toric hypersurfaces with additive actions are classified. 

We study Euler-symmetric projective toric variety. It turns out that a linearly normal projective toric variety is Euler-symmetric if and only if it admits an additive action (Theorem \ref{Maintheorem}). Also we study the set of 
Euler points on a projective toric variety (Proposition \ref{Eulerpoints}) and describe fundamental forms corresponding to linearly normal Euler-symmetric projective toric varieties (Proposition \ref{fundform}). 

The autor is grateful to Ivan Arzhantsev for stating the problem and helpful remarks.

\section{Complete toric varieties admitting additive actions}

Here we recall the description of complete toric varieties admitting additive actions obtained in \cite{ARRO}. 

Let $X$ be a complete toric variety of dimension $n$ with an acting torus $T$. Let $N$ be the lattice of one-parameter subgroups of $T$, $M$ be the dual lattice of characters and $\left<\cdot, \cdot \right>: N \times M \to \mathbb{Z}$ be the natural pairing. We denote by $N_{\BQ}$ and $M_{\BQ}$ the vector spaces $N\otimes_{\mathbb{Z}} \BQ$ and $M\otimes_{\BZ} \BQ$, respectively. 

Let $\Delta$ be the fan of polyhedral cones in $N_{\BQ}$, which corresponds to  $X$; see \cite{COX} or \cite{FULTON} for details. Let $\Delta (1) = \{ \rho_1, \ldots, \rho_r\}$ be the set of one-dimensional cones in $\Delta$ and for a cone $\sigma \in \Delta$ by    $\sigma(1)$ we mean the set of one-dimensional faces of $\sigma$. We denote the primitive lattice generator of a ray $\rho$ by $p_{\rho}$ and by $p_i$ we mean $p_{\rho_i}$. 

\begin{de} A vector $e \in M$ is  called a \emph{Demazure root} of a complete fan $\Delta$ if there is a ray $\rho\in \Delta(1)$ such that $\left<p_{\rho}, e\right> = -1$ and $\left<p_{\rho'}, e\right> \geq 0$ for all $\rho' \in \Delta(1),\ \rho' \neq \rho$.

A set of Demazure roots $e_1, \ldots e_n$ of a complete fan $\Delta$ of dimension $n$ is called a \emph{complete collection} if the rays in $\Delta(1)$ can be numbered in such a way that $\left<p_i, e_j \right> = -\delta_i^j$ for all $1 \leq i, j \leq n$ where $\delta_i^j = 1$ when $i=j$ and $\delta_i^j = 0$ when $i \neq j$.  

\end{de}

The following result makes it possible to understand by the fan $\Delta$ whether $X$ admits an additive action.

\begin{theor}\cite[Corollary 2]{ARRO}\label{arzhromask} A complete toric variety $X$ admits an additive action if and only if there is a complete collection of Demazure roots of the fan $\Delta$. 	

\end{theor}

It is easy to see that there is a complete collection of Demazure roots of the fan $\Delta$ if and only if one can order rays of the fan $\Delta$ in such a way that the primitive vectors on the first $n$ rays form a basis of the lattice $N$ and the remaining rays lie in the negative octant with respect to this basis (see Figure \ref{fig:M1}).

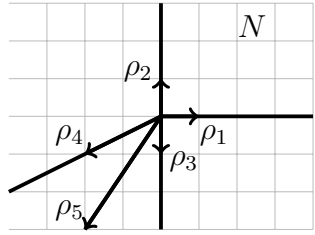
\begin{figure}[hbt!]
\centering
\begin{tikzpicture}

\draw[step = 0.5cm, lightgray] (1, -1) grid (5, 2);
\draw[line width = 0.5mm] (3, 0.5) -- (5, 0.5);
\draw[line width = 0.5mm] (3, 0.5) -- (3, 2);
\draw[line width = 0.5mm] (3, 0.5) -- (3, -1);
\draw[line width = 0.5mm] (3, 0.5) -- (1, -0.5);
\draw[line width = 0.5mm] (3, 0.5) -- (2, -1);
\draw[->, line width = 0.5mm] (3, 0.5) -- (3.5, 0.5);
\draw[->, line width = 0.5mm] (3, 0.5) -- (3, 1);
\draw[->, line width = 0.5mm] (3, 0.5) -- (3, 0);

\draw[->, line width = 0.5mm] (3, 0.5) -- (2, 0);

\draw[->, line width = 0.5mm] (3, 0.5) -- (2, -1	);

\draw (3.7, 0.25) node {$\rho_1$};
\draw (2.7, 1.1) node {$\rho_2$};
\draw (3.3, -0.1) node {$\rho_3$};
\draw (1.8, 0.25) node {$\rho_4$};
\draw (1.8, -0.75) node {$\rho_5$};

\draw(4.2, 1.7) node {$N$};

\end{tikzpicture}
\caption{The fan on this picture corresponds to a complete toric variety that admits an additive action}\label{fig:M1}
\end{figure}

\begin{de}
Let $X\subseteq \mathbb{P}^s$ be a normal nondegenerate projective variety. The embedding of $X$ into $\mathbb{P}^s$ defines a map of spaces $H^0(\mathbb{P}^s, \mathcal{O}(1)) \to H^0(X, \mathcal{O}_X(1))$. We say that $X$ is \emph{linearly normal} projective variety if this map is surjective.
\end{de}

Every nondegenerate linearly normal projective toric variety can be given by a polytope in $M_{\mathbb{Q}}$. We recall that a \emph{lattice polytope} in $M_{\mathbb{Q}}$ is a convex hull of a finite subset in $M$.

Let $m$ be a vertex of a lattice polytope $P$. Denote by $S_{P, m}$ the semigroup in $M$ generated by the set $\left(P\cap M\right) - m$. The semigroup $S_{P, m}$ is called \emph{saturated} if for all $k \in \BN\setminus \{0\}$ and for all $a\in M$ the condition $ka\in S_{P, m}$ implies $a\in S_{P, m}$. A lattice polytope $P$ is \emph{very ample} if for every vertex $m \in P$, the semigroup $S_{P, m}$ is saturated.

Let $P \subseteq M_{\mathbb{Q}}$ be a full dimensional very ample lattice polytope and $P\cap M = \{m_0, m_1, \ldots, m_{s}\}$. Then one can consider the map

$$T \longrightarrow \BP^{s},\ \ \ t \longmapsto [\chi^{m_0(t)}: \ldots : \chi^{m_{s}(t)}],$$
where $\chi^{m_i}$ is the character of $T$ corresponding to $m_i$. Denote by $X_P$ the closure of the image of this map. Then $X_P$ is a linearly normal nondegenerate projective toric variety.

Conversely, let $Y \subseteq \BP^{s}$ be a linearly normal nondegenerate projective toric variety with an acting torus $T$. Then $Y$ coincides with $X_P$ for some very ample polytope $P$ in $M$ up to automorphism of~$\BP^{s}$.

\begin{de} \label{Rectangle} A lattice polytope $P \subseteq M_{\mathbb{Q}}$ is \emph{inscribed in a rectangle} if there is a vertex $v_0 \in P$ such that

\begin{itemize}
\item[(1)] the primitive vectors on the edges of $P$ containing $v_0$ form a basis $e_1, \ldots, e_n$ of the lattice $M$;

\item[(2)] for every inequality $\left<p, x\right>\leq a$ on $P$ that corresponds to a facet of $P$ not passing through $v_0$ we have $\left<p, e_i\right>\geq 0$ for all $i = 1, \ldots, n$.
\end{itemize}

\end{de}

\begin{theor} \cite[Theorem 4]{ARRO}
Let $P\subseteq M_{\mathbb{Q}}$ be a very ample polytope and $X_P$ be the corresponding projective toric variety. Then $X_P$ admits an additive action if and only if the polytope $P$ is inscribed in a rectangle.
\end{theor}

\begin{figure}[hbt!]
\centering
\begin{tikzpicture}

\draw[step = 0.5cm, lightgray] (1, -1) grid (5, 2);
\draw[->, thick] (1, 0) -- (5, 0);
\draw[->, thick] (2, -1) -- (2, 2);
\draw[line width = 0.5mm] (2, 0) -- (4, 0) -- (3.5, 1) -- (2.5, 1.5) -- (2, 1.5) -- (2, 0) ;
\draw[->, line width = 0.5mm] (2, 0) -- (2.5, 0);
\draw[->, line width = 0.5mm] (2, 0) -- (2, 0.5);
\draw (1.75, 0.33) node {$e_2$};
\draw (2.35, -0.25) node {$e_1$};
\draw (4.25, 1.75) node {$M$};

\draw[step = 0.5cm, lightgray] (8, -1) grid (12, 2);
\draw[->, thick] (8, 0) -- (12, 0);
\draw[->, thick] (9, -1) -- (9, 2);
\draw[line width = 0.5mm] (9, 0) -- (10, 0) -- (10.5, 0.5) -- (9.5, 1.5) -- (9, 1.5) -- (9, 0) ;
\draw[->, line width = 0.5mm] (9, 0) -- (9.5, 0);
\draw[->, line width = 0.5mm] (9, 0) -- (9, 0.5);
\draw (8.75, 0.33) node {$e_2$};
\draw (9.35, -0.25) node {$e_1$};
\draw (11.25, 1.75) node {$M$};

\end{tikzpicture}
\caption{The polytope on the left is inscribed in a rectangle and the polytope on the right is not. }\label{fig:M4}
\end{figure}
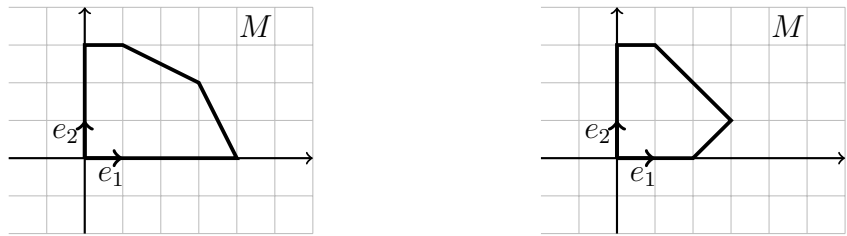

\section{Fundamental forms}

As we mentioned before an Euler-symmetric variety is uniquely determined by its fundamental form. We recall some definitions.

Let $X \subseteq \BP^s$ be a $n$-dimensional nondegenerate irreducible projective subvariety of the projective space $\BP^s$ and  $x\in X$ be a smooth point. Denote by $z_0, \ldots, z_s$ homogeneous coordinates on $\BP^s$ and by $y_i = \frac{z_i}{z_0}$ the respective inhomogeneus coordinates on the affine chart $U_0 = \{z\in \BP^s | z_0 \neq 0\}$. We may assume that $x = [1:0:\ldots :0]$ with respect to these coordinates and the tangent space $T_x X$ is given by equations $y_i = 0$ for 
$i =n+1,\ldots, s$. 

Then the functions $y_1, \ldots, y_n$ are the system of local parameters on $X$. Denote by $L$ the set of linear combinations $\{l = \alpha_0 + \alpha_1y_1 + \ldots + \alpha_s y_s\ |\ \alpha_i \in \BC \}$.

For any function $l\in L$ there is an open neighborhood of $x$ such that $l$ is given by the series 

$$l = \sum_{i=0}^{\infty} h^i_l (y_1, \ldots, y_n),$$
where $h^i_l$ is a homogeneous polynomial of degree $i$. Denote by $\overline{h}(l)$ the first non-zero term. 

\begin{de}The \emph{$k$-th fundamental form} of the variety $X$ at the point $x$ is a vector space 
$$F^k = \left<\overline{h}(l)|l\in L,\ \deg \overline{h} (l) = k \right> \subseteq \mathrm{Sym}^k (T_x X)^*.$$

The \emph{fundamental form} of the variety $X$ at the point $x$ is the vector space

$$F = \bigoplus^{\infty}_{k=0} F^{k} \subseteq \bigoplus_{k=0}^{\infty}\mathrm{Sym}^k (T_x X)^*.$$
\end{de}

\begin{de} We say that the fundamental form is \emph{monomial} if one can choose coordinates $y_1, \ldots, y_n$ such that there is a basis in $F$ consisting of monomials. Each monomial $y_1^{a_1}\ldots y_n^{a_n}$ defines a point $(a_1,\ldots, a_n) \in \BZ_{\geq 0}^n$ so a monomial fundamental form defines a set $D(F) = \{ (a_1, \ldots, a_n)\ |\ y_1^{a_1}\ldots y_n^{a_n} \in F\}$. 
\end{de}

\begin{re} One can find in \cite[Chapter 3]{IL} a more generally accepted definition of fundamental form. It follows from \cite[Lemma 2.5]{FH} that these definitions are equivalent. 

\end{re}

\begin{de}
Let $V$ be a vector space. For any vector $v\in V$ and a non-negative integer number $k$ the contraction homomorphism is a linear map  $\iota_v :\mathrm{Sum}^{k} V^* \to \mathrm{Sym}^{k-1} V^*$ which sends $\varphi \in \mathrm{Sum}^{k} V^*$ to $\iota_v (\varphi) \in \mathrm{Sym}^{k-1} V^*$ defined by

$$\iota_v (\varphi) (v_1, \ldots, v_k) = \varphi(v, v_1, \ldots, v_k),$$
for any $v_1, \ldots, v_k \in V$. By convention, we define $\iota_v (\mathrm{Sym}^0 V^*) = 0$. 

\end{de}

\begin{de} Consider a subspace $ W = \bigoplus_{k \geq 0} W^k \subseteq \bigoplus_{k \geq 0} \mathrm{Sym}^k V^*$. Then the subspace~$W$ is called a \emph{symbol system} if $W^0 = \BC, W^1 = V^*$, $W^k \neq 0$ only for finite number of $k$ and $\iota_v (W) \subseteq W$ for all $v\in V$.   

\end{de} 

It is well know that the fundamental form at a general point is a symbol system (see \cite[Chapter 3]{IL}).  It follows from \cite[Theorem 3.7]{FH} that an Euler-symmetric variety is uniquely defined by its fundamental form at a general point. Moreover, for any symbol system $F$ there is an Euler-symmetric variety $X$ such that a fundamental form of $X$ at a general point is $F$.

\section{Euler-symmetric projective toric varieties}

First of all, note that if a normal projective variety $X$ is an Euler-symmetric variety with respect to some nondegenerate embedding of $X$ into a projective space, then $X$ is Euler-symmetric with respect to any linearly normal embedding of $X$ into a projective space. More precisely, the following statement is true.

\begin{prop}\label{equiv}
Let $X$ be a normal projective variety and $x\in X$ be a smooth point. The following conditions are equivalent. 

\begin{enumerate}

\item The point $x$ is an Euler point with respect to some nondegenerate embedding of $X$ into a projective space.

\item The point $x$ is an Euler point with respect to any nondegenerate linearly normal embedding of $X$ into a projective space.

\end{enumerate}     
\end{prop}

\begin{proof}
$(1) \Rightarrow (2).$ Here we use reasoning similar to the proof of \cite[Proposition 3.2.6 ]{BRION}.

Let $X\subseteq \mathbb{P}^s$ be a nondegenerate projective variety and $x\in X$ be a smooth Euler point. Then there is a $\mathbb{G}_m$-action of Euler type at $x$. Consider a nondegenerate linearly normal embedding $\rho: X \hookrightarrow \mathbb{P}^k$. Denote by $L$ the restriction on $X$ of the line bundle $\mathcal{O}(1)$ on~$\mathbb{P}^k$. Since $\mathbb{G}_m$ is a factorial variety the line bundle $L$ can be linearizied with respect to the action of $\mathbb{G}_m$ on $X$; \cite[Proposition 2.4]{KKLV}. The linarization defines a rational linear $\mathbb{G}_m$-action on $H^0(X, L) \simeq H^0(\mathbb{P}^k, \mathcal{O}(1))$. It defines extended $\mathbb{G}_m$-action on~$\mathbb{P}^k$.  Therefore $x$ is an Euler point with respect to embedding $\rho: X \hookrightarrow \mathbb{P}^k$.

Any normal projective variety admits a nondegenerate linearly normal embedding into a projective space. So the implication  $(2) \Rightarrow (1)$ is trivial. 

\end{proof}

\begin{cor}\label{CorES}
Let $X$ be a normal projectvie variety. Then $X$ is Euler-symmetric with respect to some nondegenerate embedding into a projective space if and only if $X$ is Euler-symmetric with respect to any nondegenerate linearly normal embedding of $X$ into a projective space.
\end{cor}

Now we consider the toric case.

\begin{prop}\label{Eulerpoint}
Let $X$ be a projective toric variety with an acting torus $T$ and $x \in X$ be a smooth $T$-fixed point. Then $x$ is an Euler point with respect to any linearly normal nondegenerate embedding of $X$ into a projective space.  
\end{prop}

\begin{proof}

Let  $M$ be the lattice of characters of $T$. Any linearly normal nondegenerate embedding of $X$ into a projective space correponds to some very ample lattice polytope $P \subseteq M_{\mathbb{Q}}$.  Denote by $v_0$ the vertex of $P$ corresponding to $x$. We can assume that $v_0$ is the zero point of $M$. 

Since $x$ is a smooth point, the primitive vectors on the edges of $P$ containing $v_0$ form a basis $e_1, \ldots, e_n$. All vertices of $P$ have non-negative coordinates with respect to this basis.

Denote by $\{m_0, m_1, \ldots, m_s\}$ the lattice points in $P$. We can assume that the first $n+1$ points have the following coordinates in the basis $e_1, \ldots, e_n$:

$$m_0 = v_0 = (0, \ldots, 0), m_1 = (1, 0, \ldots, 0), \ m_2 = (0, 1, \ldots, 0), \ldots, m_n = (0, 0, \ldots, 1).$$
We denote the coordinates of the other points as:
$$m_{n+i} = (a_{i, 1}, \ldots, a_{i,n}),\ i = 1,\ldots, s-n, $$
where $a_{i,j}$ are non-negative integer numbers. 

Let $t_1, \ldots, t_n$ be the coordinates on $T$ corresponding to the basis $e_1, \ldots, e_n$. For any character $m\in M$ and for any $t\in T$ we denote by $t^m$ the value $\chi^m(t)$.  

Consider the map 
$$\varphi_D: T \to \mathbb{P}^{s},\ (t_1, \ldots, t_n) \to [t^{m_0}:\ldots: t^{m_s}] =$$ $$=  [1: t_1: \ldots: t_n: t_1^{a_{1,1}}\cdots t_n^{a_{1,n}}:\ldots :t_1^{a_{s-n,1}}\cdots t_n^{a_{s-n,n}}].$$

The variety $X$ is the closure of the image of the map $\varphi_D$. Note that the point $x$ has coordinates $[1: 0:\ldots: 0]$. 

It is easy to see that the action of $T$ on $X$ can be extended to an action of $T$ on $\BP^s$. Consider the subtorus 
$$T_1 = \{(t, \ldots, t) | t\in \BC^{\times} \}\subseteq T.$$ 

Denote by $z_0, \ldots, z_s$ the homogeneous coordinats on $\BP^s$. Consider the affine chart $U_0 = \{[z_0 : \ldots :z_s]\ |\ z_0 \neq 0\} \subseteq \BP^s$. Let $y_i = \frac{z_i}{z_0}$ be the corresponding inhomogeneous coordinates on $U_0$. Then the variety $X \cap U_0$ satisfies equations $y_{n+i} = y_1^{a_{i,1}}\ldots y_n^{a_{i,n}}.$  We see that $x = (0, \ldots, 0)$ is an isolated fixed point in $X$ with respect to the action of $T_1$. 

The tangent space at a point $x$ on $X$ is given by equations $y_{n+1} = \ldots = y_s = 0$. But $T_1$ acts by multiplication by $t$ on the variables $y_1, \ldots, y_n$. So $T_1$ acts by scalar operators on $T_x X$. So $x$ is a Euler point with respect to the action of $T_1$. 

\end{proof}

\begin{re}\label{remark}
It follows from the proof of Proposition \ref{Eulerpoint} that the fundamental form $F$ of $X$ at the point $x$ is monomial and the set $D(F)$ coincides with the set $\{m_0, m_1, \ldots, m_s\}$. 
\end{re}

\begin{prop}\label{Eulerpoints}
Let $X$ be a projective toric variety with an acting torus $T$ and $x$ be a smooth point in $X$. The following conditions are equivalent. 

\begin{enumerate}

\item The point $x$ is an Euler point with respect to some nondegenerate embedding of $X$ into a projective space.

\item The point $x$ is an Euler point with respect to any nondegenerate embedding of $X$ into a projective space.

\item There is an automorphism $\varphi$ of $X$ such that $\varphi(x)$ is a $T$-fixed point.

\end{enumerate}     
\end{prop}

\begin{proof}

We start with implication $(1) \Rightarrow (3)$. Suppose $x\in X$ is an Euler point with respect to an action of one-dimensional torus $T_1$ with respect to some nondegenerate embedding into a projective space. Denote by $G$ the connected component of unit in the automorphism group $\Aut (X)$. Then $G$ is an affine algebraic group (see \cite[Theorem 4.2]{COX2}). Hence there is a maximal subtorus $T'$ in  $G$ such that $T_1$ is a subgroup of $T'$. All points in the orbit $T'x$ are $T_1$-fixed. Since $x$ is isolated $T_1$-fixed point then $x$ is a $T'$-fixed point.

All maximal subtors in $G$ are conjugated, so there is an automorphism $\varphi \in G$ such that $\varphi T' \varphi^{-1} = T$. Then $\varphi(x)$ is a $T$-fixed point.

Now we prove $(3) \Rightarrow (2)$. Suppose that there is an automorphism $\varphi$ of $X$ such that $\varphi(x)$ is a $T$-fixed point. Then $x$ is a fixed point with respect to the torus $\varphi T \varphi^{-1}$. The variety $X$ is a toric variety with respect to the action of the torus $\varphi T \varphi^{-1}$. By Proposition \ref{Eulerpoint} the point $x$ is an Euler point with respect to any linearly normal nondegenerate embedding.

The equivalence of (1) and (2) follows from Proposition \ref{equiv}.
\end{proof}

\begin{theor}\label{Maintheorem}

Let $X$ be a projective toric variety with an acting torus $T$. The following conditions are equivalent:

\begin{enumerate}
\item $X$ is Euler-symmetric with respect to some nondegenerate embedding into a projective space;
\item $X$ is Euler-symmetric with respect to any nondegenerate linealy normal embedding into a projective space;
\item $X$ admits an addititve action.
\end{enumerate}

\end{theor}

\begin{proof}

$(1) \Rightarrow (3)$. Suppose that $X$ is Euler-symmetric with respect to some nondegenerate embedding into a projective space. Then $X$ admits an additive action by \cite[Theorem 3.7]{FH} (see also Proposition \ref{Simpleproof} below). 

$(3) \Rightarrow (2)$ Now suppose that $X$ admits an additive action. Then by \cite[Theorem 3]{ARRO} $X$ admits an additive action normalized by $T$. Let $U$ be the open orbit with respect to this additive action. The orbit $U$ is isomorphic to the affine space and $T$-invariant. Therefore the variety $U$ is an affine toric variety with respect to the action of $T$. An affine toric variety has a $T$-fixed point if and only if it has no regular invertible functions except constants. So  there is a $T$-fixed smooth point $p\in U$. By Proposition \ref{Eulerpoints} all points in $U$ are Euler points with respect to any nondegenerate linealy normal embedding into a projective spaces. 

The equivalence of (1) and (2) follows from Corollary \ref{CorES}. 
\end{proof}

It is proved in \cite{FH} that any Euler-symmetric projective vaiety admits an additive action. However, in the case of toric varieties one can prove it using the combinatorial description of toric varieties. For this purpose we need a description of the orbits of the automorphism group of complete toric varieties obtained by Ivan Bazhon \cite{Bazhov}.

For a ray $\rho_i$ we denote by $D_i$ the $T$-invariant divisor corresponding to  $\rho_i$ and by $[D_i]$ we mean the class of this divisor in the divisor class group.  For a cone $\tau \in \Delta$ we define a monoid

$$\Gamma (\tau) = \sum_{\rho_i \in \Delta(1) \setminus \tau (1)} \mathbb{Z}_{\geq 0}[D_{i}] .$$
Denote by $\Upsilon (\Delta)$ the set of monoids $\{\Gamma(\tau): \tau \in \Delta \}.$ By $O_{\tau}$ we mean the torus orbit on $X$ corresponding to $\tau \in \Delta$.

\begin{theor}\cite[Theorem 3.7]{Bazhov}\label{Bazh} Let $X$ be a complete toric variety. Torus orbits $O_{\sigma}$ and $O_{\sigma'}$ on $X$ lie in the same $\mathrm{Aut}(X)$-orbit if and only if there exists an automorphism $\phi:~\mathrm{Cl}(X)\to~\mathrm{Cl}(X)$ with the following properties:
\begin{itemize}
\item $\phi(\Gamma(\sigma)) = \Gamma (\sigma')$,
\item $\phi(\Upsilon(\Delta)) = \Upsilon(\Delta),$
\item there exists a permutation $f$ of elements $\{1, \ldots, r\}$ such that $\phi([D_i]) = [D_{f(i)}].$
\end{itemize}

\end{theor}

\begin{prop}\label{Simpleproof}
Let $X \subseteq \mathbb{P}^s$ be a nondegenerate Euler-symmetric projective toric variety. Then $X$ admits an additive action.
\end{prop}

\begin{proof}

Let $X$ be an Euler-symmetric projective toric variety with an acting torus $T$. Then there is an open subset $U$ consisting of Euler points in $X$. Denote by $O$  the open $T$-orbit in~$X$. It is clear that $O$ is a subset of $U$. By Proposition \ref {Eulerpoints} there is a $T$-fixed Euler point $x$ which belongs to the same $\mathrm{Aut}(X)$-orbit with points from $O$. 

Denote by $\Delta$ the fan corresponding  to $X$. Let $\sigma$ be the cone in $\Delta$ corresponding to the $T$-orbit $x$ and $\sigma_0$ be the cone corresponding to the $T$-orbit $O$. Note that $\sigma$ is a maximal cone in $\Delta$ and $\sigma_0$ is a vertex. 

Let $\Delta(1) = \{\rho_1, \ldots, \rho_r\}$ be the set of rays in $\Delta$ and $p_1, \ldots, p_r$ be the primitive lattice generators on these rays. We suppose that the first $n$ rays belong to $\sigma$. Since $x$ is a smooth point the vectors $p_1, \ldots, p_n$  form a basis of the lattice $N$. Then the monid $\Gamma(\sigma_0)$  is equal to $\sum^r_{i = 1}  \mathbb{Z}_{\geq 0}[D_i]$ and the monoid $\Gamma(\sigma)$ is equal to $\sum^r_{i=n+1} \mathbb{Z}_{\geq 0}[D_i]$.

By \cite[Theorem 3.7]{Bazhov} there is an automorphism  $\phi $ of the group $\Cl (X)$ such that $ \phi(\Gamma (\sigma_0)) = \Gamma (\sigma)$. The automorphism $\phi$  permutes the elements of the set $\{[D_1], \ldots, [D_r]\}$.  Since the monoid~$\Gamma( \sigma_0)$  is generated by the set $\{[D_1], \ldots, [D_r]\}$, we have $\phi(\Gamma (\sigma_0)) = \Gamma(\sigma_0)$.  It implies that $\Gamma (\sigma_0) = \Gamma (\sigma)$. It follows that for any $i = 1,\ldots, n$ there are non-negative integeres $a_{i,j}$ such that

$$[D_{i}] = \sum^r_{j=n+1} a_{i,j} [D_{j}].$$

There is an exact sequence 

$$0 \longrightarrow M \longrightarrow \BZ^r \longrightarrow \Cl (X) \longrightarrow 0,$$
where the second arrow is given by the map

$$m \to (\left<p_1, m\right>, \ldots, \left<p_r, m\right>),\ m\in M, $$
and the third arrow is given by the map

$$(b_1, \ldots, b_r) \to b_1[D_1] + \ldots + b_r[D_r],$$
where $b_j \in \mathbb{Z}$; see \cite[Chapter 4]{COX} for details. 

Since $[D_{i}] - \sum^r_{j=n+1} a_{i,j} [D_{j}] = 0$ for $i=1,\ldots n$ there are vectors $m_i \in M$ such that $\left<p_j, m_i\right> = \delta_i^j$ for $1 \leq i, j \leq n$ and $\left<p_j, m_i\right> = -a_{i,j}$, for $1 \leq i \leq n < j \leq r$. 

Then $m_1, \ldots, m_n$ is a dual basis to $p_1, \ldots, p_n$ and the vectors $p_j$ with $j > n$ have coordinates $-a_{i,j}$. So rays $\rho_{n+1}, \dots, \rho_r$ lie in the negative octant with respect to the basis $p_1, \ldots, p_n$. By Theorem \ref{arzhromask} the variety $X$ admits an additive action.

\end{proof}

It is easy to describe symbol systems which correspond to toric Euler-symmetric varieties.

\begin{prop} \label{fundform}

Let $X \subseteq \mathbb{P}^s$ be a linearly normal nondegenerate projective Euler-symmetric variety of dimension $n$ and $F$ be its fundamental form at a general point. Then $X$ is a toric variety if and only if $F$ is monomial and the set $D(F) \subseteq \mathbb{Z}_{\geq 0}^n$ coincides with the set of lattice points inside  some very ample inscribed in a rectangle lattice polytope. 
\end{prop}

\begin{proof}

Let $X \subseteq \mathbb{P}^s$ be a linearly normal nondegenerate Euler-symmetric projective toric variety with an acting torus $T$. Then $X$ admits an additive action.  The embedding of $X$ into the projective space $\mathbb{P}^s$ is given by some very ample lattice polytope $P \subseteq M_{\mathbb{Q}}$.  By \cite[Theorem 4]{ARRO} $P$ is inscribed in a rectangle.

Denote by $v_0$ the vertex of $P$ from Definition \ref{Rectangle} and by $x\in X$ the smooth $T$-fixed point corresponding to $v_0$. We can assume that $v_0$ is the origin of $M$.  The point $x$ belongs to the open $\mathbb{G}_a^n$-orbit in $X$. Since $\mathbb{G}_a^n$ is a factorial variety the action of $\mathbb{G}_a^n$ can be extended on~$\mathbb{P}^s$. So at a general point in $X$ the fundamental form is isomorphic to $F$. By Remark \ref{remark} the fundamental form $F$ at the point $x$ is monomial and the set $D(F)$ coincides with the set of lattice points inside $P$.

Conversely,  let $X \subseteq \mathbb{P}^s$ be a linearly normal nondegenrate projective Euler-symmetric variety and suppose that the fundamental form $F$ at a general point of $X$ is monomial and the set $D(F) \subseteq \mathbb{Z}^n$ coincides with the set of lattice points inside some very ample inscribed in a rectangle lattice polytope $P$. Consider the toric variety $X_P \subseteq \mathbb{P}^s$ corresponding to $P$. Then fundamental form at a general point of $X_P$ is $F$. Since Euler-symmetric variety is unquely determined by its fundamental form at a general point, the varieties $X$ and $X_{P}$ are isomorphic. So $X$ is a toric variety.
\end{proof}

At the end we consider two examples. 

\begin{ex}

As follows from Proposition \ref{Eulerpoints},  each Euler point on projective toric varieties $X$ belongs to $\mathrm{Aut}(X)$-orbit of a $T$-fixed point. At the same time two Euler points can lie in different $\mathrm{Aut}(X)$-orbits. 

Consider the Hirzebruch surface $\mathcal{H}_s$ with $s\geq 1$. It is a toric variety with the fan $\Delta_{\mathcal{H}_s}$ is shown in Figure \ref{fig:M2}.

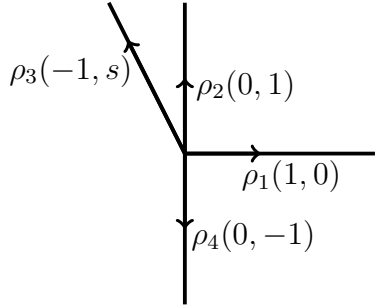
\begin{figure}[hbt!]
\centering
\begin{tikzpicture}

\draw[line width = 0.5mm] (0, 0) -- (2.5, 0);
\draw[line width = 0.5mm] (0, -2) -- (0, 2);
\draw[line width = 0.5mm] (0, 0) -- (-1, 2);

\draw[->, line width = 0.5mm] (0, 0) -- (1, 0);
\draw (1.4, -0.3) node {$\rho_1(1,0)$};

\draw[->, line width = 0.5mm] (0, 0) -- (0, 1);
\draw (0.8, 0.9) node {$\rho_2(0,1)$};

\draw[->, line width = 0.5mm] (0, 0) -- (0, -1);
\draw (0.9, -1.1) node {$\rho_4(0,-1)$};

\draw[->, line width = 0.5mm] (0, 0) -- (-0.75, 1.5);
\draw (-1.5, 1.1) node {$\rho_3(-1, s)$};

\end{tikzpicture}

\caption{The fan corresponding to Hirzerbruch surface $\mathcal{H}_s$. }\label{fig:M2}
\end{figure}

The Hirzebruch surface is smooth. It is well known that any complete toric surface is a projective variety; see \cite[Chapter II.4]{FULTON}.

By $\sigma_{ij}$ we mean a two-dimensional cone in $\Delta_{\mathcal{H}_s}$ with rays $\rho_i$ and $\rho_j$. By $\sigma_i$ we mean a one-dimensional cone $\rho_i$ and by $\sigma_0$ the vertex of $\Delta_{\mathcal{H}_s}$ respectively.

The divisor class group $\mathrm{Cl}(X)$ is freely generated by $[D_{\rho_3}]$ and $[D_{\rho_2}]$. We have  $[D_{\rho_1}] = [D_{\rho_3}]$ and $[D_{\rho_4}] = [D_{\rho_2}] + s[D_{\rho_3}]$. Then all monoids in $\Upsilon(\Delta_{\mathcal{H}_s})$  except $\Gamma(\sigma_{12}),\ \Gamma(\sigma_{23})$ and $\Gamma(\sigma_2)$ are equal to $\langle[D_{\rho_2}], [D_{\rho_3}]\rangle$, where by $\langle S\rangle$ we mean a monoid generated by a set $S$. We denote the monid $\langle[D_{\rho_2}], [D_{\rho_3}]\rangle$ by $A$.

Monoids $\Gamma(\sigma_{12}),\ \Gamma(\sigma_{23})$ and $\Gamma(\sigma_2)$ are equal to the monoid $B = \langle[D_{\rho_3}], [D_{\rho_2}] + s[D_{\rho_3}]\rangle$. The element $[D_{\rho_2}]$ does not belong to $B$, so $A \neq B$.

Any automorphism $\phi$ from \cite[Theorem 3.7]{Bazhov} permutes monoids in $\Upsilon(\Delta_{\mathcal{H}_s})$. But in $\Upsilon(\Delta_{\mathcal{H}_s})$ there are 5 monoids equal to $A$ and 3 monoids equal to $B$. So there is no such $\phi$ that $\phi(A) = B$. 

Then all points in $\mathcal{H}_s$ belong to $\mathrm{Aut}(X)$-orbit of some smooth $T$-fixed point and all points are Euler points with respect to any nondegenerate linearly normal embedding into a projective space. But there are two different $\mathrm{Aut} (X)$-orbits. 
\end{ex}

\begin{ex}
Not all smooth points on a projective toric variety  are necessarily Euler points. Any  projective toric variety that does not admit an additive action is suitable as a counterexample. Indeed, points from the open $T$-orbit are smooth but not Euler. However it is also interesting to consider an example of Euler-symmetric variety, in which not all smooth points are Euler.   

Consider a blow-up of $\mathbb{P}^1\times \mathbb{P}^1$ at a point. It is a smooth projective toric variety and the corresponding fan is shown in Figure \ref{fig:M3}.

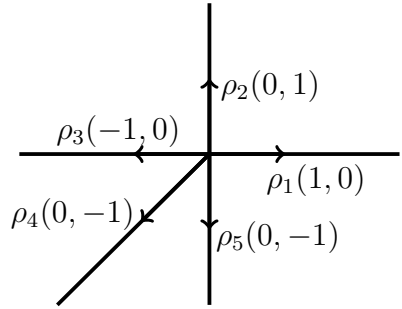
\begin{figure}[hbt!]
\centering
\begin{tikzpicture}

\draw[line width = 0.5mm] (-2.5, 0) -- (2.5, 0);
\draw[line width = 0.5mm] (0, -2) -- (0, 2);
\draw[line width = 0.5mm] (0, 0) -- (-2, -2);

\draw[->, line width = 0.5mm] (0, 0) -- (1, 0);
\draw (1.4, -0.35) node {$\rho_1(1,0)$};

\draw[->, line width = 0.5mm] (0, 0) -- (0, 1);
\draw (0.8, 0.9) node {$\rho_2(0,1)$};

\draw[->, line width = 0.5mm] (0, 0) -- (-1, 0);
\draw (-1.2, 0.3) node {$\rho_3(-1, 0)$};

\draw[->, line width = 0.5mm] (0, 0) -- (-0.9, -0.9);
\draw (-1.8, -0.8) node {$\rho_4(0,-1)$};

\draw[->, line width = 0.5mm] (0, 0) -- (0, -1);
\draw (0.9, -1.1) node {$\rho_5(0,-1)$};

\end{tikzpicture}

\caption{The fan corresponding to  blow-up of $\mathbb{P}^1\times \mathbb{P}^1$ at a point. }\label{fig:M3}
\end{figure}

The divisor class group $\mathrm{Cl} (X)$ is freely generated by classes $[D_{\rho_3}], [D_{\rho_4}]$ and $[D_{\rho_5}]$. Also we have $[D_{\rho_1}] = [D_{\rho_3}]+[D_{\rho_4}]$ and $[D_{\rho_2}] = [D_{\rho_4}] + [D_{\rho_5}]$. Let us consider $[D_{\rho_3}], [D_{\rho_4}], [D_{\rho_5}]$ as a basis of $\mathrm{Cl} (X)$. We will show that the points from the orbit corresponding to $\sigma_4$ are not Euler. Indeed, the monoid $\Gamma(\sigma_4) = \langle(1, 0, 0), (0, 0, 1), (1, 1, 0), (0, 1, 1)\rangle$ is not equal to any of monoids corresponding to fixed points listed below: 

$$\Gamma(\sigma_{23}) = \langle(0,1,0), (0,0,1),(1,1,0)\rangle,\ \ \Gamma(\sigma_{15}) = \langle(0,1,0), (1,0,0),(0,1,1)\rangle,$$
$$\Gamma(\sigma_{34}) = \langle(0,0,1), (1,1,0), (0,1,1)\rangle, \  \Gamma(\sigma_{45}) = \langle(1,0,0),(1,1,0), (0,1,1)\rangle,$$
$$ \Gamma(\sigma_{12}) = \langle(1, 0, 0), (0, 1, 0), (0, 0, 1)\rangle. $$

Any automorphism $\phi$ of the group $\mathrm{Cl} (X)$ satisfying the condition of Theorem \ref{Bazh} permutes the elements $[D_{\rho_i}]$. So $\phi$ is an automorphism of monoid $\langle(1,0,0), (0,1,0), (0,0,1)\rangle$. The elements $[D_{\rho_3}], [D_{\rho_4}], [D_{\rho_5}]$ are irreducible in this monoid but $[D_{\rho_1}]$ and $[D_{\rho_2}]$ are not. So $\phi$ permutes the elements $[D_{\rho_3}], [D_{\rho_4}], [D_{\rho_5}]$. The elements $[D_{\rho_1}]$ and $[D_{\rho_2}]$ are both divisible by $[D_{\rho_4}]$. So $\phi$ is either identical or $\phi$ permutes $[D_{\rho_3}]$ and $[D_{\rho_5}]$ and preserves $[D_{\rho_4}]$. In both cases  $\Gamma(\sigma_4)$ is preserved. 

Therefore $\mathrm{Aut} (X)$-orbit of any $T$-fixed point does not meet points in $O_{\sigma_4}$. So points in $O_{\sigma_4}$ are not Euler but smooth. 

\end{ex}


\begin{thebibliography}{99}

\bibitem{A}
Ivan Arzhantsev. Flag varieties as equivariant compactifications of $\mathbb{G}_a^n$ . Proc. Amer. Math. Soc. 139
(2011), no. 3, 783–786

\bibitem{ABZ}
Ivan Arzhantsev, Sergey Bragin, and Yulia Zaitseva. Commutative algebraic monoid structures on affne spaces. Commun. Contemp. Math. 22, (2020), no. 8, P. 1950064: 1.

\bibitem{AP}
Ivan Arzhantsev and Andrei Popovkiy. Additive actions on projective hypersurfaces. Automorphisms in Birational and Affne Geometry, Springer Proc. Math. Stat., 79, (2014), 17-33

\bibitem{AS}
Ivan Arzhantsev and Elena Sharoyko. Hassett-Tschinkel correspondence: Modality and projective hypersurfaces. J. Algebra 348 (2011), no. 1, 217-232

\bibitem{ARRO}
 Ivan Arzhantsev and Elena Romaskevich. Additive actions on toric varieties. Proc. Amer. Math. Soc. 145 (2017), no. 5, 1865-1879

\bibitem{Bazhov}
Ivan Bazhov. On orbits of the automorphism group on a complete toric variety. Beitr. Algebra Geom. 54 (2013), no. 2, 471–481.

\bibitem{BRION} 
Michel Brion. Linearization of algebraic group actions.  In Handbook of group actions. Adv. Lect. Math. 41 (2018) , 291–340.

\bibitem{COX}
David Cox, John Little, and Henry Schenck. Toric Varieties. Grad. Stud. Math. 124, AMS,
Providence, RI, 2011

\bibitem{COX2}
David Cox. The homogeneous coordinate ring of a toric variety. J. Algebraic Geom. 4 (1995), no. 1, 17–50

\bibitem{D}
Rostislav Devyatov. Unipotent commutative group actions on flag varieties and nilpotent multiplications. Transform. Groups 20 (2015), no. 1, 21–64

\bibitem{DL}
Ulrich Derenthal and Daniel Loughran. Singular del Pezzo surfaces that are equivariant compactifications. J. Math. Sci. 171 (2010), no. 6, 714–724

\bibitem{DZHUNUS1}
Sergey Dzhunusov. Additive actions on complete toric surfaces. Internat. J. Algebra Comput.  31 (2021), no. 1, 19-35 

\bibitem{DZHUNUS2}
Sergey Dzhunusov. On uniqueness of additive actions on complete toric varieties. arXiv:2007.10113


\bibitem{F}
Evgeny Feigin. $\mathbb{G}^M_a$ degeneration of flag varieties. Selecta Math. 18 (2012), no. 3, 513–537

\bibitem{FULTON}
William Fulton. Introduction to toric varieties. Ann. of Math. Stud. 131, Princeton University Press, Princeton, NJ, 1993

\bibitem{FH}
Baohua Fu and Jun-Muk Hwang. Euler-symmetric projective varieties. Algebr. Geom. 7 (2020), no. 3, 377–389

\bibitem{HT}
Brendan Hassett and Yuri Tschinkel. Geometry of equivariant compactifications of $\mathbb{G}_n^a $. Int. Math. Res. Not. IMRN  1999 (1999), no. 22, 1211–1230

\bibitem{IL}
Thomas Ivey and Joseph Landsberg. Cartan for beginners: differential geometry via moving frames and exterior differential systems. Grad. Stud. Math 61. American Mathematical Society, Providence, RI, 2003

\bibitem{KL}
Friedrich Knop and Herbert Lange. Commutative algebraic groups and intersections of quadrics. Math. Ann. 267 (1984), no. 4, 555-571

\bibitem{KKLV}
Friedrich Knop, Hanspeter Kraft , Domingo Luna, Thierry Vust.  Local properties of algebraic group actions. Invariant Theory Algebr. Transform. Groups 13 (1989), 63–75

\bibitem{S}
Anton Shafarevich. Additive Actions on Toric Projective Hypersurfaces. Results Math. 76 (2021), no. 145.



\end{thebibliography}
\end{document}